\documentclass{DongLiu}
 \usepackage{lineno,hyperref}
 \usepackage{amsmath,amssymb,amsthm}
 \usepackage{mathrsfs}
 \usepackage{color}
 \usepackage{epstopdf}
 \usepackage{amsfonts}
 \usepackage{float}
 \usepackage{cases}
 \usepackage{cite}
\newtheorem{theorem}{Theorem}[section]

\newtheorem{cor}{Corollary}[section]

\newtheorem{lemma}[theorem]{Lemma}

\newtheorem{defi}[theorem]{Definition}
\newtheorem*{tha}{Theorem A}

\newtheorem*{th1}{Theorem I}
\newtheorem*{th2}{Theorem II}

\begin{document}

%
%
%
%
%
%
%
%
%

\title[Unicity problem on complete K\"ahler manifolds]
 {Unicity problem on meromorphic mappings of complete K\"ahler manifolds}


\author[Xianjing Dong]{Xianjing Dong}

\address{%
School of Mathematical Sciences\\
Qufu Normal University\\
Qufu 273165\\
Shandong\\
People's Republic of China}

\email{xjdong05@126.com}

\author[Mengyue Liu]{Mengyue Liu$^*$}

\address{%
School of Mathematical Sciences\\
Qufu Normal University\\
Qufu 273165\\
Shandong\\
People's Republic of China}

\email{mengyueliumath@163.com}

\thanks{The research work is supported by the Natural Science Foundation of Shandong Province of China
 (ZR202211290346).}


\date{January 1, 2004}

\begin{abstract}
Nevanlinna's unicity theorems have always held an important position in value distribution theory. The  main purpose  of this paper is to 
generalize the classical Nevanlinna's unicity theorems to  non-compact complete K\"ahler manifolds with non-positive sectional curvature or nonnegative Ricci curvature.
 \end{abstract}

\maketitle

\section{Introduction}

\quad\quad
Nevanlinna's five-value theorem (see, e.g., \cite{Hu, Hay, Yang}) is well-known as a famous theorem in value distribution theory stated as follows.
\begin{tha}[Nevanlinna]
Let $f_1, f_2$ be two nonconstant meromorphic functions on $\mathbb C.$ Let $a_1,\cdots,a_5$ be five distinct values in $\overline{\mathbb C}.$
If $f_1, f_2$ share $a_j$ ignoring multiplicities for $j=1,\cdots,5,$ then $f_1\equiv f_2.$
\end{tha}

Theorem A  was generalized    by many authors. For instance,  C.-C.  Yang \cite{Yang} weakened the condition for sharing
five values to ``partially" sharing five values;
Li-Qiao \cite{Qiao} extended    it to five small functions by replacing five values; G. Valiron \cite{Vali}  generalized  it to algebroid functions;
W. Stoll \cite{Sto} considered  some related problems  when domains are parabolic manifolds; Y. Aihara \cite{A1, A3,Ai} extended it to  meromorphic functions on a finite ramified  covering space of $\mathbb C^m;$ etc.  
We refer the reader to more  related literature such as
Dulock-Ru \cite{Ru},
 H. Fujimoto \cite{Fji},  S. Ji \cite{ji, Ji} and M. Ru \cite{ru}, etc.
In addition, 5$IM$ problem  also generalizes  to $3CM+1IM$ and $2CM+2IM$ problems,  refer to   Hu-Li-Yang \cite{Hu} and Yang-Yi \cite{Yang}, etc.

     There seem to be few results  regarding   the unicity problem   on a general complex manifold. For example, is the   five-value theorem true on a complete K\"ahler manifolds with non-negative Ricci curvature (refer to examples of such manifolds in \cite{S-Y, T-Y1, T-Y2})?
  In this paper, 
 we shall consider two types of K\"ahler manifolds, which have  non-positive sectional curvature or  non-negative Ricci curvature.  
  The main trick  is
to employ   Carlson-Griffiths theory  developed by the first author \cite{Dong, Dong2}. 

 Let $M$ be a   non-compact complete K\"ahler manifold.  For    a meromorphic function $f$ on $M,$ we have the characteristic function $T_f(r, \omega_{FS})$ of $f$ under two kinds of  curvature conditions (see  the definition in Section 2).
Let $\kappa(r)$ be the Ricci term defined by (\ref{ka}) in Section 2. 

 The first main result is  as follows. 

\begin{th1}[=Theorem \ref{im1}]\label{dde3}
 Let $f_1, f_2$ be two nonconstant meromorphic functions on $M.$ Let $a_1,\cdots,a_5$ be five distinct values in $\overline{\mathbb C}.$ If $f_1, f_2$ share $a_j$ ignoring multiplicities for $j=1,\cdots,5,$ then each of the following conditions ensures that  $f_1\equiv f_2${\rm:}

$(a)$ 
$M$ has non-negative Ricci curvature and carries a positive global Green function{\rm;}  
 
 $(b)$ 
 $M$ has non-positive sectional curvature and 
 $f_1, f_2$  satisfy the growth condition
$$ \liminf_{r\rightarrow\infty}\frac{\kappa(r)r^2}{T_{f_j}(r, \omega_{FS})}=0, \ \ \ j=1, 2.$$
\end{th1}

Assume that $M$ has non-negative Ricci curvature.  Note that  $M$ carries a  positive global Green function if and only if 
$M$ satisfies the volume growth condition
$$\int_1^\infty\frac{t}{V(x, t)}dt<\infty$$
for some  $x\in M,$ where $V(x, r)$ denotes the Riemannian volume of a geodesic ball centered at $x$ with radius $r$ in $M.$  In particular,  a  positive global Green function exists 
 if $M$ has  maximal volume growth, i.e., 
$$\liminf_{r\to\infty}\frac{V(x, r)}{r^{\dim_{\mathbb R}M}}>0.$$

Let $D=\sum_{j}\mu_{j}D_{j}$ be an effective divisor,  where $D_{j}^,s$ are prime divisors. Set 
$$\mu_{D}= \inf_{j}\big\{\mu_{j}\big\}.$$
We obtain  the following result.   

\begin{th2}[=Theorem \ref{cm1}]\label{123}
 Let $f_1, f_2$ be two nonconstant meromorphic functions on $M.$ Let $a_1,\cdots,a_q$ be distinct values in $\overline{\mathbb C}.$ Let $\beta>0$ be an integer such that $q>2\beta^{-1}+2.$  If 
 $f_{1}^{*}a_{j}=f_{2}^{*}a_{j}$ with $\mu_{f_{i}^{*}a_{j}}\ge \beta$ for $i=1,2$ and $j=1,\cdots,q,$   then  each of the following conditions ensures that $f_{1}\equiv f_{2}${\rm:}

$(a)$ 
 $M$ has non-negative Ricci curvature and carries a positive global Green function{\rm;} 
 
 $(b)$ 
 $M$ has non-positive sectional curvature and
 $f_1, f_2$  satisfy the growth condition
$$ \liminf_{r\rightarrow\infty}\frac{\kappa(r)r^2}{T_{f_j}(r, \omega_{FS})}=0, \ \ \ j=1, 2.$$
\end{th2}

\section{Carlson-Griffiths theory}

\quad\quad
 Let $(M,g)$ be a $m$-dimensional  non-compact complete K\"ahler manifold  with K\"ahler form 
$$\alpha=\frac{\sqrt{-1}}{\pi}\sum_{i,j=1}^mg_{i\bar{j}}dz_i\wedge d\bar{z}_j$$
in a local holomorphic coordinate $(z_1,\cdots,z_m).$ 
 Let $X$ be a complex projective manifold of complex dimension not greater  than  $m.$  Put a  positive Hermitian line bunlde $(L, h)$ over $X$  such that the Chern form
$c_1(L,h)=-dd^c\log h>0,$ where 
\begin{equation*}
d=\partial+\overline{\partial}, \ \ \  d^{c}=\frac{\sqrt{-1}}{4\pi}(\overline{\partial}-\partial) \ \  \  \text{so that} \ \ \  dd^{c}=\frac{\sqrt{-1}}{2\pi}\partial\overline{\partial}.
\end{equation*}
 Let  $s_D$ be the   section  associated to $D\in|L|,$ namely, a holomorphic section of $L$ over $X$ with zero divisor  $D.$  
 Let 
 $$K_X=\bigwedge^m T^*X$$ 
 be the canonical line bundle over $X,$  where $T^*X$ is the holomorphic cotangent bundle over $X.$
Fix a reference point $o\in M.$ 
Let $f:M\to X$ be a meromorphic mapping.
In the following, we shall  define Nevanlinna's functions with  different curvature conditions. 

$1^\circ$  \emph{$M$ has non-negative Ricci curvature}

Assume that $M$ carries a positive global Green function, i.e., it satisfies the volume growth condition:
$$\int_1^\infty\frac{t}{V(t)}dt<\infty,$$
where $V(r)$ is the Riemannian volume of the geodesic ball centered at $o$ with radius $r.$ Then,  there exists a unique minimal positive global Green function $G(o, x)$ for $M$ satisfying
$$-\frac{1}{2}\Delta G(o,x)=\delta_o(x),$$
where $\Delta$ denotes the Laplace-Beltrami operator and $\delta_o$ is the Dirac function with a pole at $o.$ 
This implies that there exist constants $B>A>0$  such that 
$$A\int_{\rho(x)}^\infty\frac{t}{V(t)}dt\le G(o,x)\le B\int_{\rho(x)}^\infty\frac{t}{V(t)}dt, \ \ \ ^\forall x\in M,$$
where $\rho(x)$ is the Riemannian distance function of $x$ from $o$ (see Li-Yau\cite{LY}).
Define $$\Delta(r)=\left\{x\in M: \ G(o,x)>A\int_r^\infty \frac{t}{V(t)}dt\right\}, \ \ \ ^\forall r>0.$$
It is clearly that $\Delta(r)$ is relatively compact for all $r>0,$ and the sequence $\{\Delta(r_n)\}_{n=1}^\infty$ exhausts $M$ if $$0<r_1<r_2<\cdots<r_n<\cdots\to \infty.$$
In further, the Sard's theorem implies that the boundary $\partial\Delta(r)$ of $\Delta(r)$ is a submanifold of $M$ for almost all $r>0.$
Set $$g_r(o, x)=G(o, x)-A\int_r^\infty \frac{t}{V(t)}dt,$$
which defines the Green function of $\Delta/2$ for $\Delta(r)$ with a pole at $o$ satisfying Diricheler boundary condition, i.e.,
$$-\frac{1}{2}\Delta g_r(o,x)=\delta_o(x), \ \ \ ^\forall x\in\Delta(r); \ \ \ g_r(o,x)=0, \ \ \ ^\forall x\in\partial\Delta(r).$$
Furthermore, $g_{r}(o, x)$ defines the harmonic measure $\pi_r$ on $\partial \Delta(r),$ i.e.,
      \begin{equation*}\label{Har}
d\pi_{r}(x)=\frac{1}{2}\frac{\partial g_{r}(o, x)}{\partial \vec\nu}d\sigma_{r}(x), \ \ \  ^\forall x\in\partial \Delta(r),
\end{equation*}
where  $\partial/\partial \vec\nu$ is the inward  normal derivative on $\partial \Delta(r),$   $d\sigma_r$ is the Riemannian area element of $\partial \Delta(r).$

The Nevanlinna's functions (\emph{characteristic function}, \emph{proximity function} and \emph{counting function} as well as  \emph{simple counting function}) are   respectively defined by
  \begin{align*}
T_f(r, L)&=-\frac{1}{4}\int_{\Delta(r)}g_r(o,x)\Delta \log(h\circ f)dv, \\
m_f(r,D)&=\int_{\partial \Delta(r)}\log\frac{1}{\|s_D\circ f\|}d\pi_r, \\
N_f(r,D)&=\frac{\pi^m}{(m-1)!}\int_{f^*D\cap\Delta(r)}g_r(o,x)\alpha^{m-1}, \\
\overline{N}_f(r, D)&=\frac{\pi^m}{(m-1)!}\int_{{\rm Supp} (f^*D)\cap\Delta(r)}g_r(o,x)\alpha^{m-1},
 \end{align*}
where  $dv$ is the Riemannian volume element of $M.$

 Recently, the first author \cite{Dong2} obtained the following first main theorem and second main theorem.
\begin{theorem}[Dong, \cite{Dong2}]\label{first2}
Let $f: M\to X$ be a meromorphic mapping such that  $f(o)\not\in {\rm{Supp}}D.$ Then
$$m_f(r,D)+N_f(r,D)=T_f(r,L)+O(1).$$
\end{theorem}

\begin{theorem}[Dong, \cite{Dong2}]\label{second2}
Let $f:M\rightarrow X$ be a differentiably non-degenerate meromorphic mapping 
and let $D\in|L|$ be a reduced divisor of simple normal crossing type. Then  for any $\delta>0,$ there exists a subset $E_\delta\subseteq(0, \infty)$ of finite Lebesgue measure such that
 $$T_f(r,L)+T_f(r, K_X)
\leq  \overline N_f(r,D)+O\left(\log^+T_f(r,L)+\delta\log r\right)$$
holds for all $r>0$ outside $E_\delta.$
\end{theorem}

$2^\circ$   \emph{$M$ has non-positive sectional curvature.} 

Denote by $B(r)$  the geodesic ball in $M$ centered at $o$ with radius $r$ and by $\partial B(r)$ the geodesic sphere  centered at $o$ with radius $r.$ Then, it follows from Sard's theorem that $\partial B(r)$ is a submanifold of $M$ for almost any $r>0.$
In further, denote by $g_r(o,x)$  the positive Green function of $\Delta/2$  for $B(r),$ with a pole at $o$  satisfying  Dirichlet boundary condition.  Then, $g_r(o,x)$ defines  the harmonic measure  $\pi_r$ on  $\partial B(r)$ with respect to $o.$ 
 Let $Ric$  be the Ricci curvature tensor  of $M.$
Set
\begin{equation}\label{ka}
  \kappa(r)=\frac{1}{2m-1}\inf_{x\in B(r)}\mathcal{R}(x),
\end{equation}
where $\mathcal{R}$ is the
pointwise lower bound of Ricci curvature defined by
$$\mathcal{R}(x)=\inf_{\xi\in T_{x}M, \ \|\xi\|=1} Ric(\xi,\bar{\xi}).$$

 The Nevanlinna's functions (\emph{characteristic function}, \emph{proximity function} and \emph{counting function} as well as  \emph{simple counting function}) are   defined by
  \begin{align*}
T_f(r, L)&=-\frac{1}{4}\int_{B(r)}g_r(o,x)\Delta \log(h\circ f)dv, \\
m_f(r,D)&=\int_{\partial B(r)}\log\frac{1}{\|s_D\circ f\|}d\pi_r, \\
N_f(r,D)&= \frac{\pi^m}{(m-1)!}\int_{f^*D\cap B(r)}g_r(o,x)\alpha^{m-1}, \\
\overline{N}_f(r, D)&= \frac{\pi^m}{(m-1)!}\int_{{\rm Supp}(f^*D)\cap B(r)}g_r(o,x)\alpha^{m-1},
 \end{align*}
respectively, where  $dv$ is the Riemannian volume element of $M.$

In 2023,  the first author \cite{Dong}  gave an extension of   Carlson-Griffiths theory (see \cite{gri, gri1}) to a non-positively curved complete K\"ahler manifold, namely, who obtained the following first main theorem and second main theorem.
\begin{theorem}[Dong, \cite{Dong}]\label{first}  
Let $f: M\rightarrow X$ be a meromorphic mapping such that  $f(o)\not\in {\rm{Supp}}D.$ Then
$$m_f(r,D)+N_f(r,D)=T_f(r,L)+O(1).$$
\end{theorem}

\begin{theorem}[Dong, \cite{Dong}]\label{main}  
  Let $f:M\rightarrow X$ be a differentiably non-degenerate meromorphic mapping and let $D\in|L|$ be a divisor of simple normal crossing type.
  Then  for any $\delta>0,$ there exists a subset $E_\delta\subseteq(0, \infty)$ of finite Lebesgue measure such that
  \begin{eqnarray*}
 T_f(r,L)+T_f(r, K_X)
\leq  \overline N_f(r,D)+O\left(\log^+T_f(r,L)-\kappa(r)r^2+\delta\log r\right)
 \end{eqnarray*}
holds for all $r>0$ outside $E_\delta.$
\end{theorem}

Now, we consider several defect relations.
 Define the \emph{defect} $\delta_f(D)$ and the \emph{simple defect} $\bar\delta_f(D)$ of $f$ with respect to $D,$ respectively   by
 \begin{align*}
\delta_f(D)&=1-\limsup_{r\rightarrow\infty}\frac{N_f(r,D)}{T_f(r,L)}, \\
 \bar\delta_f(D)&=1-\limsup_{r\rightarrow\infty}\frac{\overline{N}_f(r,D)}{T_f(r,L)}.
 \end{align*}
Using  the first main theorem (cf. Theorem \ref{first2} or Theorem \ref{first}),  we have  $$0\leq \delta_f(D)\leq\bar\delta_f(D)\leq 1.$$

For any two  holomorphic line bundles $L_1, L_2$ over $X,$  define  (see \cite{gri, gri1})
$$\left[\frac{c_1(L_2)}{c_1(L_1)}\right]=\inf\left\{s\in\mathbb R: \ \omega_2<s\omega_1;  \ ^\exists\omega_1\in c_1(L_1),\  ^\exists\omega_2\in c_1(L_2) \right\},$$
where $c_1(L_j)$ denotes  the first Chern class of $L_j$ for $j=1,2.$

\begin{cor}[Defect relation]\label{dde}  
Let $f:M\rightarrow X$ be a differentiably non-degenerate meromorphic mapping and let $D\in|L|$ be a divisor of simple normal crossing type. Then each of the following conditions ensures that
 $$\delta_f(D)\leq\bar\delta_f(D)\leq  \left[\frac{c_1(K_X^*)}{c_1(L)}\right]:$$

$(a)$ \ 
$M$ has non-negative Ricci curvature and carries a positive global Green function{\rm;}

$(b)$ \
$M$ has non-positive sectional  curvature and $f$ satisfies the growth condition
$$ \liminf_{r\rightarrow\infty}\frac{\kappa(r)r^2}{T_f(r, L)}=0.$$ 
\end{cor}

A $\mathbb Q$-line bundle is  an element in ${\rm{Pic}}(M)\otimes\mathbb Q,$  where  ${\rm{Pic}}(M)$ denotes  the Picard group over $M.$
Let $F\in{\rm{Pic}}(M)\otimes\mathbb Q$ be a $\mathbb Q$-line bundle.    $F$ is said to be
\emph{ample} (resp.  \emph{big}), if $\nu F\in {\rm{Pic}}(M)$
is ample (resp.  big) for some positive  integer $\nu.$
Define
$$T_f(r,F)=\frac{1}{\nu}T_f(r,\nu F),$$
where $\nu$ is a positive integer such that $\nu F\in {\rm{Pic}}(X).$ Evidently, this is well defined.
  For a holomorphic line bundle $F$ over $X,$  define
$$\left[\frac{F}{L}\right]=\inf\left\{\gamma\in\mathbb Q: \gamma L\otimes F^{-1} \ \text{is big}\right\}.$$
It is easy  to see that  $[F/L]<0$ if and only if $F^{-1}$ is big.
 
\begin{cor}[Defect relation]\label{defect}
 Assume  the same conditions as in  Corollary  {\rm\ref{dde}}. Then each of  $(a)$ and $(b)$ in  Corollary  {\rm\ref{dde}} ensures that 
 \begin{equation*}
\delta_f(D)\leq\bar\delta_f(D)\leq\left[\frac{K^{-1}_X}{L}\right].
 \end{equation*}
\end{cor}

\begin{proof}  It follows from the definition of $[K_X^{-1}/L]$  that $([K_X^{-1}/L]+\epsilon)L\otimes K_X$ is big for any rational number  $\epsilon>0$.  Then, we obtain
$$\left(\left[K^{-1}_X/L\right]+\epsilon\right)L\otimes K_X\geq\delta L$$
for a sufficiently small  rational number $\delta>0.$ This implies that
$$T_f(r,K^{-1}_X)\leq \left(\left[K^{-1}_X/L\right]-\delta+\epsilon\right)T_f(r,L)+O(1).$$
By Theorem \ref{second2} (resp. Theorem \ref{main}), we conclude that
$$\delta_f(D)\leq\bar\delta_f(D)\leq
\left[\frac{K^{-1}_X}{L}\right].$$
\end{proof}
\begin{theorem}\label{cor1}  
 Let $f:M\rightarrow X$ be a differentiably non-degenerate meromorphic mapping.
Assume that  $\mu F\otimes L^{-1}$ is big for some positive integer $\mu,$ where $F$ is  a big line bundle and $L$ is a holomorphic line bundle over $X.$ Then
$$T_f(r, L)\leq \mu T_f(r, F)+O(1).$$
\end{theorem}
\begin{proof}
The bigness of $\mu F\otimes L^{-1}$  implies that there exists a nonzero holomorphic section $s\in H^0(X, \nu(\mu F\otimes L^{-1}))$ for a sufficiently large positive integer $\nu.$ 
By Theorem \ref{first2} (resp. Theorem \ref{first}), we have
  \begin{align*}
N_f(r, (s))&\leq T_f(r,\nu(\mu F\otimes L^{-1}))+O(1)\\
&= \mu\nu T_f(r, F)-\nu T_f(r,L)+O(1).
  \end{align*}
  This leads to the desired inequality.
\end{proof}

\section{Propagation of algebraic dependence}
\quad\quad
Let $M$ be a non-compact complete K\"ahler manifold with complex dimension $m$ and let $X$ be a complex projective manifold with complex dimension not higher than  $m.$ 
Fix an integer $l\geq2.$
A proper algebraic subset $\Sigma$ of $X^l$ is said to be  \emph{decomposible}, if there exist  $s$ positive integers $l_1,\cdots,l_s$
with $l=l_1+\cdots+l_s$ for some  integer $s\leq l$
and algebraic subsets $\Sigma_j\subseteq X^{l_j}$ for $1\leq j\leq s,$ such that
$\Sigma=\Sigma_1\times\cdots\times\Sigma_s.$ If $\Sigma$ is not decomposable,  we say that $\Sigma$ is \emph{indecomposable.}
For $l$ meromorphic mappings $f_1,\cdots,f_l: M\rightarrow X,$
 there is  a meromorphic mapping
 $f_1\times\cdots\times f_l: M\rightarrow X^l,$ defined by
 $$(f_1\times\cdots\times f_l)(x)=\big(f_1(x),\cdots, f_l(x)\big), \ \ \  ^{\forall} x\in M\setminus \bigcup_{j=1}^lI(f_j),$$
 where $I(f_j)$ denotes the indeterminacy set of $f_j$ for $1\leq j\leq l.$
As a matter of convenience, set
$$\tilde f=f_1\times\cdots\times f_l.$$
\begin{defi} Let $S$ be an analytic subset of  $M.$  The nonconstant meromorphic mappings $f_1,\cdots,f_l: M\rightarrow X$ are said to be algebraically dependent on $S,$ if there exists a proper indecomposable algebraic subset $\Sigma$ of $X^l$ such that $\tilde f(S)\subseteq\Sigma.$ In this case, we  say that $f_1,\cdots,f_l$ are $\Sigma$-related on $S.$
\end{defi}

Let $L$ be a positive line bundle over $X,$ and let $D_1,\cdots,D_q\in |L|$ such that $D_1+\cdots+D_q$ has only  simple normal crossings.
   Set
\begin{equation*}
\mathscr Y=\big\{\text{$f: M\rightarrow X$ is a differentiably non-degenerate meromorphic mapping}\big\}.
\end{equation*}
Let $S_1,\cdots, S_q$ be hypersurfaces of $M$ such that $\dim_{\mathbb{C}}S_{i}\cap S_{j}\le m-2$ if $m\ge2$ or $S_{i}\cap S_{j}=\emptyset$ if $m=1$ for all $i\ne j.$
  Let $\tilde L$ be a big line bundle over $X^l.$ In general, we have
$$\tilde L\not\in \pi^*_1{\rm{Pic}}(X)\oplus\cdots\oplus\pi^*_l{\rm{Pic}}(X),$$
where $\pi_k:X^l\rightarrow X$ is the natural projection on the $k$-th factor for $1\leq k\leq l.$
Let $F_1,\cdots,F_l$   be big line bundles over $X.$ Then, it defines a line bundle over $X^l$ by
$$\tilde F=\pi^*_1F_1\otimes\cdots\otimes\pi^*_lF_l.$$
If  $\tilde L\not=\tilde F,$  we  assume that there is a  rational number $\tilde\gamma>0$ such that $$\tilde\gamma\tilde F\otimes\tilde L^{-1} \ \text{is big}.$$
  If
$\tilde L=\tilde F,$  we shall take $\tilde\gamma=1.$ In further,    assume that there is  a line bundle $F_0\in\{F_1,\cdots,F_l\}$ such that $F_0\otimes F_j^{-1}$ is either  big or trivial for $1\leq j\leq l.$

Let $\mathscr H$ be the set of all indecomposable  hypersurfaces $\Sigma$ in $X^l$ satisfying $\Sigma={\rm{Supp}}\tilde D$ for some $\tilde D\in|\tilde L|.$  
\begin{defi}
Let $D=\sum_{j}\mu_{j}D_{j}$ be an effective divisor,  where $D_{j}$ are prime divisors. Define
\begin{equation*}\label{mu}
\mu_{D}= \inf_{j}\big\{\mu_{j}\big\}.
\end{equation*}
\end{defi}

We introduce the notations $\mathscr F, \mathscr F_\kappa$ and $\mathscr G, \mathscr G_\kappa$ as follows.

 $1^\circ$   \emph{$M$ has non-negative Ricci curvature.}
 
Assume that $M$ carries a positive global Green function. Denote by
\begin{equation*}
\mathscr F=\mathscr F\big(f\in \mathscr Y; (M, \{S_j\}); (X, \{D_j\})\big)
\end{equation*}
the set of all $f\in\mathscr Y$ satisfying
\begin{equation}\label{sj}
S_j={\rm{Supp}}f^*D_j, \ \ \ 1\leq j\leq q.
\end{equation}
Moreover, denote by $$\mathscr{G} =\mathscr{G}  \big(f\in \mathscr Y;  \ \{\mu_{S_{j}}\}; \ (M,\{S_{j}\}); \  (N,\{D_{j} \})\big) $$
the set of all $f\in \mathscr Y$ satisfying
\begin{equation}\label{sjj}
S_{j}= f^{*}D_{j}, \ \ \mu_{S_{j}}\ge \beta, \ \ 1\le j\le q,
\end{equation}
where $\beta$ is a positive integer. 

 $2^\circ$  \emph{ $M$ has non-positive sectional curvature.}
 
 Denote by
\begin{equation*}
\mathscr F_\kappa=\mathscr F_\kappa\big(f\in \mathscr Y; (M, \{S_j\}); (X, \{D_j\})\big)
\end{equation*}
the set of all $f\in\mathscr Y$ satisfying $(\ref{sj})$
and 
\begin{equation}\label{grow}
 \liminf_{r\rightarrow\infty}\frac{\kappa(r)r^2}{T_f(r, L)}=0.
 \end{equation}
Moreover,  denote by
$$\mathscr{G}_{\kappa} =\mathscr{G}_{\kappa}  \big(f\in \mathscr Y;  \ \{\mu_{S_{j}}\}; \ (M,\{S_{j}\}); \  (N,\{D_{j} \})\big) $$
the set of all $f\in \mathscr Y$ satisfying (\ref{sjj}) and (\ref{grow}).

In what follows, we give two propagation theorems of algebraic dependence of $l$ meromorphic mappings $f_1,\cdots,f_l$  on $M.$
Firstly, we consider the case that  each $f_j$ satisfies $(\ref{sj}).$
Set $$S=S_1\cup\cdots\cup S_q.$$  

\begin{lemma}\label{lem1} Let $f_1,\cdots,f_l\in\mathscr F\left({\rm resp.} \ \mathscr F_\kappa\right).$
Assume that $f_1,\cdots,f_l$ are $\Sigma$-related on $S$ and $\tilde f(M)\not\subseteq \Sigma$ for some $\Sigma\in\mathscr H.$ Then
$$N(r, S)\leq\tilde\gamma\sum_{j=1}^lT_{f_j}(r, F_j)+O(1) \leq \tilde\gamma\sum_{j=1}^lT_{f_j}(r, F_0)+O(1).$$
\end{lemma}
\begin{proof}  Take $\tilde D\in|\tilde L|$ such that $\Sigma={\rm{Supp}}\tilde D.$  As mentioned  earlier,  $\tilde\gamma\tilde F\otimes\tilde L^{-1}$ is big for $\tilde\gamma\not=1$ and trivial for  $\tilde\gamma=1.$   Then, by conditions with Theorem \ref{first2} ({\rm resp.}  Theorem \ref{first}) and Theorem \ref{cor1}, we conclude that 
  \begin{align*}
N(r, S) &\leq T_{\tilde f}(r, \tilde L)+O(1) \\
&\leq\tilde\gamma T_{\tilde f}(r, \tilde F)+O(1) \\
&\leq \tilde\gamma\sum_{j=1}^l T_{f_j}(r, F_j)+O(1) \\
&\leq \tilde\gamma\sum_{j=1}^lT_{f_j}(r, F_0)+O(1).
  \end{align*}
  The proof is completed.
\end{proof}

Define
\begin{equation}\label{L0}
L_0=qL\otimes\left(-\tilde\gamma lF_0\right).
\end{equation}
Again, set
$$T(r, Q)=\sum_{j=1}^lT_{f_j}(r, Q)$$
for an arbitrary  $\mathbb Q$-line bundle $Q\in{\rm{Pic}}(X)\otimes\mathbb Q.$
\begin{theorem}\label{uni1} 
 Let $f_1,\cdots,f_l\in\mathscr F\left({\rm resp.} \ \mathscr F_\kappa\right).$ Assume that $f_1,\cdots,f_l$  are $\Sigma$-related on $S$ for some $\Sigma\in \mathscr H.$ If $L_0\otimes K_X$ is big, then $f_1,\cdots,f_l$  are $\Sigma$-related on $M.$
\end{theorem}

\begin{proof}
It suffices to prove $\tilde f(M)\subseteq\Sigma.$ Otherwise,  we  assume that $\tilde f(M)\not\subseteq\Sigma.$ According to Theorem \ref{second2} ({\rm resp.} Theorem \ref{main}),  for $i=1,\cdots, l$ and $j=1,\cdots, q$
$$T_{f_i}(r, L)+T_{f_i}(r, K_X)
\leq \overline{N}_{f_i}(r, D_j)+o\big{(}T_{f_i}(r,L)\big{)},
$$
which follows  from $S_j={\rm{Supp}}f_i^*D_j$ with $1\leq i\leq l$  and $1\leq j\leq q$ that
$$
qT_{f_i}(r, L)+T_{f_i}(r, K_X)
\leq N(r, S)+o\big{(}T_{f_i}(r,L)\big{)}.
$$
Using Lemma \ref{lem1}, then
  \begin{align*}
qT_{f_i}(r, L)+T_{f_i}(r, K_X)
&\leq \tilde\gamma\sum_{i=1}^lT_{f_i}(r, F_0)+o\big{(}T_{f_i}(r,L)\big{)} \\
&= \tilde\gamma T(r, F_0)+o\big{(}T_{f_i}(r,L)\big{)}.
  \end{align*}
  Thus, we get
$$
qT(r, L)+T(r, K_X)
\leq
  \tilde\gamma l T(r, F_0)+o\big{(}T(r,L)\big{)}.
$$
  It yields that
  \begin{equation}\label{3}
  T(r, L_0)+T(r,K_X)\leq o\big{(}T(r,L)\big{)}.
  \end{equation}
  On the other hand, the bigness of $L_0\otimes K_X$ implies that there exists  a positive integer $\mu$
  such that $\mu(L_0\otimes K_X)\otimes L^{-1}$ is  big. By Theorem \ref{cor1}
  $$T(r,L)\leq \mu\big(T(r,L_0)+T(r,K_X)\big)+O(1),$$
  which contradicts with (\ref{3}).  Therefore,  we have $\tilde f(M)\subseteq\Sigma.$
\end{proof}

Set
$$\gamma_0=\left[\frac{L_0^{-1}\otimes K^{-1}_X}{L}\right],$$
where $L_0$ is defined by (\ref{L0}).  Note that  $L_0\otimes K_X$ is big if and only if $\gamma_0<0.$ Thus, it yields that
\begin{cor} 
 Let $f_1,\cdots,f_l\in\mathscr F\left({\rm resp.} \ \mathscr F_\kappa\right).$ Assume that $f_1,\cdots,f_l$  are $\Sigma$-related on $S$ for some $\Sigma\in \mathscr H.$ If $\gamma_0<0,$  then $f_1,\cdots,f_l$  are $\Sigma$-related on $M.$
\end{cor}

Now, we consider the case that each $f_j$ satisfies (\ref{sjj}).
Set $$S=S_1+\cdots+ S_q.$$
Carrying the arguments in the proof of Lemma \ref{lem1} to the situation where $f_{1},\cdots,f_{l}\in \mathscr{G}
\left({\rm resp.} \ \mathscr{G}_{\kappa}\right),$ we can easily show without any details that
\begin{lemma}\label{lem4.1}
Let $f_{1},\cdots,f_{l}\in \mathscr{G}
\left({\rm resp.} \ \mathscr{G}_{\kappa}\right).$  Assume that $f_1,\cdots,f_l$ are $\Sigma$-related on $S$ and $\tilde{f}(M) \not\subseteq \Sigma$ for some $\Sigma \in \mathscr{H}.$  Then
$$N(r,S)\le\tilde{\gamma} \sum_{j=1}^{l} T_{f_{j}}\big(r,F_{0}\big)+O(1).$$
\end{lemma}

Define \begin{equation}\label{G0}
G_{0} =qL\otimes \big(-\beta^{-1}\tilde{\gamma}lF_{0}\big).
\end{equation}

\begin{theorem}\label{uni2}
Let $f_{1},\cdots,f_{l}\in \mathscr{G}
\left({\rm resp.} \ \mathscr{G}_{\kappa}\right).$  Assume that  $f_{1},\cdots,f_{l}$ are $\Sigma$-related on $S$ for some $\Sigma \in \mathscr{H}.$ If $G_{0}\otimes K_{X}$ is big$,$ then $f_{1},\cdots,f_{l}$ are $\Sigma$-related on $M.$
\end{theorem}

\begin{proof}
It suffices to prove  $\tilde{f}(M)\subseteq \Sigma.$ Assume the contrary that $\tilde{f}(M)\not\subseteq \Sigma$. According to Theorem \ref{second2} ({\rm resp.} Theorem \ref{main}),  for $j=1,\cdots,l$
\begin{align*}
qT_{f_{j}}(r,L)+T_{f_{j}}(r,K_{X})&\le \overline{N}(r,S)+o\big(T_{f_{j}}(r,L)\big)\\
&\le\beta^{-1}N(r,S)+o\big(T_{f_{j}}(r,L)\big),
\end{align*}
which follows from Lemma \ref{lem4.1} that
\begin{equation*}
qT_{f_{j}}(r,L)+T_{f_{j}}(r,K_{X})\le \beta^{-1}\tilde{\gamma}\sum_{j=1}^{l} T_{f_{j}}\big(r,F_{0}\big)+o\big(T_{f_{j}}(r,L)\big).
\end{equation*}
Thus, we conclude that
\begin{equation*}
qT(r,L)+T(r,K_{X})\le \beta^{-1}\tilde{\gamma}l T\big(r,F_{0}\big)+o\big(T(r,L)\big).
\end{equation*}
It yields that
\begin{equation*}\label{e15}
T(r,G_{0})+T(r,K_{X})\le o\big(T(r,L)\big),
\end{equation*}
which is a contradiction since the bigness of $G_{0} \otimes K_{X}.$ Therefore, we have $\tilde{f}(M)\subseteq \Sigma$.
\end{proof}

\section{Nevanlinna's unicity theorems}
\quad\quad
We use the same notations as in Section 3.
  Since $X$ is projective,  there is a  holomorphic embedding $\Phi: X \hookrightarrow\mathbb P^N(\mathbb C).$
 Let $\mathscr O(1)$ be the hyperplane line bundle over $ \mathbb P^N(\mathbb C).$ Take $l=2$ and $F_1=F_2=\Phi^*\mathscr O(1),$ then  it follows that  $F_0=\Phi^*\mathscr O(1)$ and
 $$\tilde F=\pi_1^*\left(\Phi^*\mathscr O(1)\right)\otimes \pi_2^*\left(\Phi^*\mathscr O(1)\right).$$
 Again, set $\tilde L=\tilde F,$ then $\tilde\gamma=1.$
 In view of (\ref{L0}), we  have
$$
 L_0=qL\otimes\left(-2\Phi^*\mathscr O(1)\right).
$$
In addition, with the aid of  (\ref{G0}), we also have $$G_{0}=qL\otimes\big(-2\beta^{-1}\Phi^{*} \mathscr{O}(1)\big).$$

 Fix a
 $f_0\in\mathscr F\left({\rm resp.} \ \mathscr F_\kappa\right).$ Denote by  $\mathscr F_0 \left({\rm resp.} \ \mathscr F_{\kappa, 0}\right)$  the set of all meromorphic mappings $f\in\mathscr F\left({\rm resp.} \ \mathscr F_\kappa\right)$ such that $f=f_0$ on the   hypersuface $S,$ where  $S=S_1\cup\cdots\cup S_q.$ 
 Similarly,
 fix a
 $\tilde{f}_0\in\mathscr G\left({\rm resp.} \ \mathscr G_\kappa\right).$ Denote by  $\mathscr G_0 \left({\rm resp.} \ \mathscr G_{\kappa, 0}\right)$  the set of all meromorphic mappings $f\in\mathscr G\left({\rm resp.} \ \mathscr G_\kappa\right)$ such that $f=\tilde{f}_0$ on the   hypersuface $S,$ where  $S=S_1+\cdots+ S_q.$
 
 \begin{lemma}\label{t1}
We have
 
 $(a)$  
 If  $L_0\otimes K_X$ is big,  then $\mathscr F_0 \left({\rm resp.} \ \mathscr F_{\kappa,0}\right)$ has only  one element{\rm;}
 
 $(b)$  
  If  $G_0\otimes K_X$ is big, then $\mathscr G_0 \left({\rm resp.} \ \mathscr G_{\kappa,0}\right)$ has only  one element.
 \end{lemma}

 \begin{proof}  
 For $(a),$
  it suffices to show that $f\equiv f_0$ for all $f\in\mathscr F_0 \left({\rm resp.} \ \mathscr F_{\kappa,0}\right).$
 Recall that  $\Phi: X \hookrightarrow\mathbb P^N(\mathbb C)$ is a   holomorphic embedding.
 Since  $f=f_0$ on $S,$  we have   $\Phi\circ f=\Phi\circ f_0$ on $S.$
First, we  assert  that
 $\Phi\circ f\equiv\Phi\circ f_0.$
   Otherwise,   we may assume that $\Phi\circ f\not\equiv\Phi\circ f_0.$
 Let $\Delta$ denote the diagonal of $\mathbb P^N(\mathbb C)\times \mathbb P^N(\mathbb C).$ Put $\tilde \Phi=\Phi\times \Phi$
 and $\tilde f=f\times f_0.$ Then, it gives   a meromorphic mapping
$$\phi=\tilde \Phi\circ \tilde f:=\Phi\circ f\times \Phi\circ f_0: \  M\rightarrow \mathbb P^N(\mathbb C)\times \mathbb P^N(\mathbb C).$$
Again, define $\tilde{\mathscr O}(1):=\pi_1^*\mathscr O(1)\otimes\pi_2^*\mathscr O(1),$ which is a holomorphic line bundle over $\mathbb P^N(\mathbb C)\times \mathbb P^N(\mathbb C),$ where $\mathscr O(1)$ is the hyperplane line bundle over  $\mathbb P^N(\mathbb C).$ From the  assumption, we have
$\tilde L=\pi_1^*\left(\Phi^*\mathscr O(1)\right)\otimes \pi_2^*\left(\Phi^*\mathscr O(1)\right).$  Since  $\Phi\circ f\not\equiv\Phi\circ f_0,$  then there exists  a  holomorphic  section $\tilde\sigma$ of $\tilde{\mathscr O}(1)$ over $\mathbb P^N(\mathbb C)\times \mathbb P^N(\mathbb C)$ such  that $\phi^*\tilde\sigma\not=0$ and $\Delta\subseteq{\rm{Supp}}(\tilde\sigma).$
 Take $\Sigma={\rm{Supp}}\tilde\Phi^*(\tilde\sigma),$ then we have    $\tilde f(S)\subseteq\Sigma$ and $\tilde f(M)\not\subseteq\Sigma.$
On the other hand,  with the aid of  Theorem \ref{uni1}, the bigness of  $L_0\otimes K_X$ gives that  $\tilde f(M)\subseteq\Sigma,$ which is a contradiction.   Hence, we obtain  $\Phi\circ f\equiv\Phi\circ f_0.$
Next, we prove $f\equiv f_0.$ Otherwise, we have $f(x_0)\not=f_0(x_0)$ for some $x_0\in M\setminus I(f_0).$  However, it contradicts with
$\Phi(f(x_0))=\Phi(f_0(x_0))$ since $\Phi$ is injective. This proves $(a).$ 
For $(b),$ using the similar methods as above,  then it follows from Theorem \ref{uni2} that $f\equiv \tilde{f}_0$ for all $f\in\mathscr G_0 \left({\rm resp.} \ \mathscr G_{\kappa,0}\right).$ Hence, we prove the lemma.
 \end{proof}

 \begin{theorem}\label{im1} 
Let $f_1, f_2$ be two nonconstant meromorphic  functions on $M.$ Let $a_1,\cdots,a_q$ be distinct values in $\overline{\mathbb C}.$
Assume  ${\rm{Supp}}f_1^*a_j={\rm{Supp}}f_2^*a_j\not=\emptyset$ for  $j=1,\cdots,q.$  If $q\ge5,$ then each of the following conditions ensures $f_1\equiv f_2${\rm :}

$(a)$   $M$ has non-negative Ricci curvature and carries a positive global Green function{\rm;} 
 
 $(b)$ 
 $M$ has non-positive sectional curvature and $f_1, f_2$
satisfy the growth condition
$$ \liminf_{r\rightarrow\infty}\frac{\kappa(r)r^2}{T_{f_j}(r, \omega_{FS})}=0, \ \ \ j=1,2.$$
  \end{theorem}
 \begin{proof}
 Set $X=\mathbb P^1(\mathbb C)$ and $L=\mathscr O(1).$   Note that  $K_{\mathbb P^1(\mathbb C)}=-2\mathscr O(1),$ then
  $$L_0\otimes K_{\mathbb P^1(\mathbb C)}=q\mathscr O(1)\otimes(-2\mathscr O(1))\otimes(-2\mathscr O(1))=(q-4)\mathscr O(1).$$
 Hence,  $L_0\otimes K_{\mathbb P^1(\mathbb C)}$ is big for  $q\geq5.$ By Lemma \ref{t1}, we prove the theorem.
 \end{proof}

  \begin{cor} Let $f_1, f_2$ be two nonconstant meromorphic  functions on $\mathbb C^m.$  Let $a_1,\cdots,a_q$ be distinct values in $\overline{\mathbb C}.$
Assume that ${\rm{Supp}}f_1^*a_j={\rm{Supp}}f_2^*a_j\not=\emptyset$ for  $j=1,\cdots,q.$ If $q\geq 5,$ then $f_1\equiv f_2.$
  \end{cor}

\begin{theorem}\label{cm1}
 Let $f_1, f_2$ be two nonconstant meromorphic  functions on $M.$  Let $a_1,\cdots,a_q$ be distinct values in $\overline{\mathbb C}.$
 Assume that $ f_{1}^{*}a_{j}=f_{2}^{*}a_{j}$ with $\mu_{f_{i}^{*}a_{j}}\ge \beta$ for $i=1,2$ and $j=1,\cdots,q,$  where $\beta$ is a positive integer. If $q>2\beta^{-1}+2,$ then
 each of the following conditions ensures $f_1\equiv f_2${\rm :}

$(a)$   $M$ has non-negative Ricci curvature and carries a positive global Green function{\rm;}
 
 $(b)$ 
 $M$ has non-positive sectional curvature and $f_1, f_2$
satisfy the growth condition
$$ \liminf_{r\rightarrow\infty}\frac{\kappa(r)r^2}{T_{f_j}(r, \omega_{FS})}=0, \ \ \ j=1,2.$$
  \end{theorem}

\begin{proof}
Set $X=\mathbb{P}^{1}(\mathbb{C})$ and $L=\mathscr{O}(1)$. Note that $F_{0}=\mathscr{O}(1)$, then
\begin{align*}
G_{0}\otimes K_{\mathbb{P}^{1}}(\mathbb{C)}&=q\mathscr{O}(1)\otimes\big(-2\beta^{-1}\mathscr{O}(1)\big)\otimes\big(-2\mathscr{O}(1)\big)\\
&=\big(q-2\beta^{-1}-2\big)\mathscr{O}(1).
\end{align*}
Hence, $G_{0}\otimes K_{\mathbb{P}^{1}}(\mathbb{C)}$ is big for $q> 2\beta^{-1}+2$. By Lemma \ref{t1}, we prove the theorem.
\end{proof}

\vskip\baselineskip

\end{document}